\documentclass[reqno, twoside]{amsart}

\usepackage{amsmath}
\usepackage{amssymb}
\usepackage{graphicx}
\usepackage{subfig}
\usepackage{enumerate}
\usepackage[all]{xy}
\usepackage{hyperref}

\hypersetup{%
pdftitle={Smoothly embedded rational homology balls},%
pdfauthor={Heesang Park, Dongsoo Shin, Jongil Park},%
pdfkeywords={flip, rational blow-down/blow-up, rational homology ball},%
}

\usepackage{comment}
\usepackage{xspace}
\usepackage{mathtools}

\usepackage{mathptmx}       
\usepackage{helvet}         
\usepackage{courier}        

\usepackage{tikz}

\usetikzlibrary{patterns,decorations.pathreplacing}

\tikzset{bullet/.style={
shape = circle,fill = black, inner sep = 0pt, outer sep = 0pt, minimum size = 0.35em, line width = 0pt, draw=black!100}}

\tikzset{rectangle/.style={
shape = rectangle,fill = white, inner sep = 0pt, outer sep = 0pt, minimum size = 0.35em, line width = 0pt, draw=black!100}}

\tikzset{empty/.style={
shape = circle,fill = white, inner sep = 0pt, outer sep = 0pt, minimum size = 0.35em, line width = 0pt, draw=white!100}}

\tikzset{xmark/.style={
shape = x,fill = white, inner sep = 0pt, outer sep = 0pt, minimum size = 0em, line width = 0pt, draw=white!100}}

\tikzset{longrectangle/.style={
inner sep = 1em,
rectangle,
minimum size=1em,
very thick,
draw=black!100, 
}}

\tikzset{font=\scriptsize}

\usepackage{tikz-cd}
\usepackage{tkz-graph}

\newtheorem{theorem}{Theorem}[section]

\newtheorem{proposition}[theorem]{Proposition}
\newtheorem{corollary}[theorem]{Corollary}
\newtheorem{observation}[theorem]{Observation}
\newtheorem*{maintheorem}{Main Theorem}

\theoremstyle{definition}

\newtheorem{definition}[theorem]{Definition}
\newtheorem{remark}[theorem]{Remark}

\newtheorem{example}[theorem]{Example}

\numberwithin{equation}{section}

\newcommand{\abs}[1]{\ensuremath \left\lvert #1 \right\rvert}

\begin{document}

\title[Smoothly embedded rational homology balls]{Smoothly embedded rational homology balls}

\author[H. Park]{Heesang Park}

\address{Department of Mathematics, Konkuk University, Seoul 143-701, Korea}

\email{HeesangPark@konkuk.ac.kr}

\author[J. Park]{Jongil Park}

\address{Department of Mathematical Sciences, Seoul National University, Seoul 151-747, Korea \& Korea Institute for Advanced Study, Seoul 130-722, Korea}

\email{jipark@snu.ac.kr}

\author[D. Shin]{Dongsoo Shin}

\address{Department of Mathematics, Chungnam National University, Daejeon 305-764, Korea}

\email{dsshin@cnu.ac.kr}




\begin{abstract}
In this paper we prove the existence of rational homology balls smoothly embedded in regular neighborhoods of certain linear chains of smooth $2$-spheres by using techniques from minimal model program for 3-dimensional complex algebraic variety.
\end{abstract}

\maketitle

\section{Introduction}

A rational homology ball $B_{p,q}$ ($1 \le q < p$) is a smooth $4$-manifold with the lens space $L(p^2, pq-1)$ as its boundary such that $H_{\ast}(B_{p,q};\mathbb{Q}) \cong H_{\ast}(B^4;\mathbb{Q})$. It appears in a \emph{rational blow-down surgery} (which was developed by Fintushel--Stern~\cite{Fintushel-Stern-1997} and generalized by J. Park~\cite{JPark-1997}): If $C_{p,q}$ is a regular neighborhood of the linear chain of smooth $2$-spheres whose  dual graph is
\begin{equation}\label{equation:Cpq}
\begin{tikzpicture}
\node[bullet] (10) at (1,0) [label=above:{$-b_1$}] {};
\node[bullet] (20) at (2,0) [label=above:{$-b_2$}] {};

\node[empty] (250) at (2.5,0) [] {};
\node[empty] (30) at (3,0) [] {};

\node[bullet] (350) at (3.5,0) [label=above:{$-b_{r-1}$}] {};
\node[bullet] (450) at (4.5,0) [label=above:{$-b_r$}] {};

\draw [-] (10)--(20);
\draw [-] (20)--(250);
\draw [dotted] (20)--(350);
\draw [-] (30)--(350);
\draw [-] (350)--(450);
\end{tikzpicture}
\end{equation}
with
\begin{equation*}
\frac{p^2}{pq-1} = b_1 - \cfrac{1}{b_2 - \cfrac{1}{b_3 - \cfrac{1}{\ddots - \cfrac{1}{b_r}}}}
\end{equation*}
for $b_i \ge 2$ ($ 1 \leq i \leq r$) in a smooth $4$-manifold $X$, then one may cut $C_{p,q}$ from $X$ and paste $B_{p,q}$ along the boundary $L(p^2,pq-1)$ so that one obtains a new smooth $4$-manifold $Z=(X-C_{p,q}) \cup_{L(p^2,pq-1)} B_{p,q}$.

A rational homology ball $B_{p,q}$ itself can be also regarded as the Milnor fiber of a cyclic quotient singularity of type $\frac{1}{p^2}(1,pq-1)$ (See Section~\ref{subsection:classT} for details). So one may interpret a rational blow-down surgery as a global smoothing of a singular complex surface $X$ with a cyclic quotient singularity $o \in X$ of type $\frac{1}{p^2}(1,pq-1)$ under certain mild conditions. Explicitly, if there is no local-to-global obstruction to deform $X$, i.e., if the obstruction $H^2(X,\mathcal{T}_X)$ vanishes, then there is a $\mathbb{Q}$-Gorenstein smoothing $\pi \colon \mathcal{X} \to \Delta$ of $X$ over a small disk $\Delta = \{t \in \mathbb{C} : \abs{t} < \epsilon \}$ that is induced from a local $\mathbb{Q}$-Gorenstein smoothing of the singularity $o$.

The rational blow-down surgery was very successful for constructing small exotic $4$-manifolds (cf.~\cite{JPark-2005}, \cite{Stipsicz-Szabo-2005}, \cite{JPark-Stisicz-Szabo-2005}), and its counter part, $\mathbb{Q}$-Gorenstein smoothing, in the category of complex surface was also very useful for constructing complex surfaces of general type with small geometric genus (cf.~\cite{Lee-Park-2007}, \cite{Lee-Park-2008}, \cite{PPS-K3-2009}, \cite{PPS-K4-2009}, \cite{PPS-pg1-2013}, \cite{PPS-H1Z4-2013}, \cite{PSU-2013}).

In this way, one can easily produce a smooth $4$-manifold $Z$, where $B_{p,q}$ is smoothly embedded, by performing a rational blow-down surgery to a given $4$-manifold $X$ containing a configuration $C_{p,q}$. However, conversely, detecting $B_{p,q}$ in a given $4$-manifold $Z$ is difficult unless one knows a priori that $Z$ is obtained by performing a rational blow-down surgery.

In this paper we show that, for any relatively prime integers $p, q$ with $1 \le q < p$, there is a smoothly embedded rational homology ball $B_{p,q}$ in a regular neighborhood of a certain linear chain of smooth $2$-spheres corresponding to $p^2/(pq-1)$, which is called the \emph{$\delta$-half linear chain} (See Definition~\ref{definition:delta-half-linear-chain}). That is,

\begin{maintheorem}[Corollary~\ref{corollary:main}]
Suppose $Z$ is a smooth $4$-manifold which contains the $\delta$-half linear chain corresponding to $p^2/(pq-1)$ with $1 \le q < p$. Then there is a smoothly embedded rational homology ball $B_{p,q}$ in $Z$.
\end{maintheorem}

Hence one can detect a smoothly embedded $B_{p,q}$ in a given smooth $4$-manifold $Z$ if it contains the $\delta$-half linear chain corresponding to $p^2/(pq-1)$, which is a generalization of the recent result of Khodorovskiy~\cite{Khodorovskiy-2014}. In particular, we conclude that, in a regular neighborhood of a smooth $(-4)$-sphere, there is a rational homology ball $B_{n,1}$ for any odd $n \ge 3$; Proposition~\ref{proposition:Bn-for-V4}. The main tools of the proof are techniques from explicit semi-stable minimal model program for $3$-dimensional complex algebraic variety.

This paper is organized as follows: In Section~\ref{section:linear-chain} we briefly recall some basic notions related to Hirzebruch--Jung continued fractions and rational homology balls. And then we introduce techniques from minimal model program of 3-fold and we prove the main technical result in Section~\ref{section:flip}. The proofs of the existence of rational homology balls are given in Section~\ref{section:embedded-QHB}. Finally, we investigate the possibility of rational blow-up surgery using the embedded rational homology balls in Section~\ref{section:rational-blowup}.

\subsection*{Acknowledgements}

The authors would like to thank G. Urz\'ua for many valuable conversations. HP and DS thank Korea Institute for Advanced Study when they were associate members in KIAS.
HP was supported by Basic Science Research Program through the National Research Foundation of Korea (NRF) grant funded by the Korean Government (2011-0012111).
JP was supported by Leaders Research Grant funded by Seoul National University and by the National Research Foundation of Korea Grant (2010-0019516). He also holds a joint appointment at KIAS and in the Research Institute of Mathematics, SNU.
DS was supported by the research fund of Chungnam National University in 2014.

\section{Linear chains of projective lines and rational homology balls}
\label{section:linear-chain}

We briefly recall some relevant notions of Hirzebruch--Jung continued fractions and linear chains of $2$-spheres related to rational homology balls.

\subsection{Hirzebruch--Jung continued fractions}

The \emph{Hirzebruch--Jung continued fraction} $[b_1,\dotsc,b_r]$ for $b_i \in \mathbb{N}$ and $b_i \ge 1$ ($1 \leq i \leq r$) is defined recursively as follows: $[b_r]=b_r$ and
\[[b_i, \dotsc, b_r]=b_i-\frac{1}{[b_{i+1},\dotsc, b_r]}\]
for $i < r$. For any integers $n, a \in \mathbb{N}$ with $1 \le a < n$, it is known that $n/a$ is uniquely represented as $[b_1, \dotsc, b_r]$ for some $b_i \in \mathbb{N}$ with $b_i \ge 2$ ($i=1,\dotsc,r$). Usually, the Hirzebruch--Jung continued fraction $[b_1, \dotsc, b_r]$ represents a linear chain of $2$-spheres whose dual graph is given by
\begin{equation*}
\begin{tikzpicture}
\node[bullet] (10) at (1,0) [label=above:{$-b_1$}] {};
\node[bullet] (20) at (2,0) [label=above:{$-b_2$}] {};

\node[empty] (250) at (2.5,0) [] {};
\node[empty] (30) at (3,0) [] {};

\node[bullet] (350) at (3.5,0) [label=above:{$-b_{r-1}$}] {};
\node[bullet] (450) at (4.5,0) [label=above:{$-b_r$}] {};

\draw [-] (10)--(20);
\draw [-] (20)--(250);
\draw [dotted] (20)--(350);
\draw [-] (30)--(350);
\draw [-] (350)--(450);
\end{tikzpicture}
\end{equation*}
For $n/a=[b_1,\dotsc,b_r]$ with $b_i \ge 2$, we call the linear chain above by \emph{the linear chain corresponding to $n/a$}.

Two Hirzebruch--Jung continued fractions $n/a=[b_1,\dotsc,b_r]$ and $n/(n-a)=[a_1,\dotsc,a_e]$ are \emph{dual} to each other. That is, one can compute $a_j$'s for $n/(n-a)$ via $b_i$'s for $n/a$, and vice versa, by Riemenschneider's dot diagram \cite{Riemenschneider-1974}: Place in the $i$-th row $b_i-1$ dots, the first one under the last one of the $(i-1)$-st row; then, the column $j$ contains $a_j-1$ dots, and vice versa. Furthermore, we have
\begin{equation*}
[b_1,\dotsc,b_r,1,a_e,\dotsc,a_1]=0,
\end{equation*}
which implies that the linear chain of $2$-spheres
\begin{equation*}
\begin{tikzpicture}
\node[bullet] (10) at (1,0) [label=above:{$-b_1$}] {};
\node[bullet] (20) at (2,0) [label=above:{$-b_2$}] {};

\node[empty] (250) at (2.5,0) [] {};
\node[empty] (30) at (3,0) [] {};

\node[bullet] (350) at (3.5,0) [label=above:{$-b_{r-1}$}] {};
\node[bullet] (450) at (4.5,0) [label=above:{$-b_r$}] {};

\node[bullet] (550) at (5.5,0) [label=above:{$-1$}] {};

\node[bullet] (650) at (6.5,0) [label=above:{$-a_e$}] {};
\node[empty] (70) at (7,0) [] {};
\node[empty] (750) at (7.5,0) [] {};
\node[bullet] (80) at (8,0) [label=above:{$-a_2$}] {};
\node[bullet] (90) at (9,0) [label=above:{$-a_1$}] {};

\draw [-] (10)--(20);
\draw [-] (20)--(250);
\draw [dotted] (20)--(350);
\draw [-] (30)--(350);
\draw [-] (350)--(450);
\draw [-] (450)--(550);

\draw [-] (550)--(650);
\draw [-] (650)--(70);
\draw [dotted] (650)--(80);
\draw [-] (750)--(80);
\draw [-] (80)--(90);
\end{tikzpicture}
\end{equation*}
is blown down to \begin{tikzpicture}
\node[bullet] (00) at (0,0) [label=above:{0}] {};
\end{tikzpicture}.

\begin{example}\label{example:cyclic-49/34}
For $n=49$ and $a=34$, we have
\begin{equation*}
\text{$\dfrac{n}{a}=[2,2,5,4]$ and $\dfrac{n}{n-a}=[4,2,2,3,2,2]$}
\end{equation*}
The linear chain corresponding to $49/34$ is
\begin{equation*}
\begin{tikzpicture}
\node[bullet] (00) at (0,0) [label=above:$-2$] {};
\node[bullet] (10) at (1,0) [label=above:$-2$] {};
\node[bullet] (20) at (2,0) [label=above:$-5$] {};
\node[bullet] (30) at (3,0) [label=above:$-4$] {};

\draw [-] (00)--(10)--(20)--(30);
\end{tikzpicture}
\end{equation*}
and its Riemenschneider's dot diagram is the following:
\begin{equation*}
\begin{tikzpicture}[scale=0.5]
\node[bullet] at (0,3) [label=above:{}] {};

\node[bullet] at (0,2) [] {};

\node[bullet] at (0,1) [] {};
\node[bullet] at (1,1) [] {};
\node[bullet] at (2,1) [] {};
\node[bullet] at (3,1) [] {};

\node[bullet] at (3,0) [] {};
\node[bullet] at (4,0) [] {};
\node[bullet] at (5,0) [] {};
\end{tikzpicture}
\end{equation*}
\end{example}

\subsection{Linear chains of class $W$}

For simplicity, we define a special linear chain whose boundary bounds a rational homology ball.

\begin{definition}
A linear chain of \emph{class $W$} is a linear chain of $2$-spheres  corresponding to
\begin{equation*}
\frac{p^2}{pq-1} = [b_1,\dotsc,b_r]
\end{equation*}
with $1 \le q < p$ and $(p,q)=1$.
\end{definition}

Every linear chain of class $W$ is obtained by the following recursive algorithm:

\begin{proposition}[Wahl~\cite{Wahl-1981}]
\label{proposition:linear-chain-of-class-T}
\hfill
\begin{enumerate}
\item The linear chain \begin{tikzpicture}
\node[bullet] (00) at (0,0) [label=above:{$-4$}] {};
\end{tikzpicture} is of class $W$

\item If the linear chain
\begin{tikzpicture}
\node[bullet] (20) at (2,0) [label=above:{$-b_1$}] {};

\node[empty] (250) at (2.5,0) [] {};
\node[empty] (30) at (3,0) [] {};

\node[bullet] (350) at (3.5,0) [label=above:{$-b_r$}] {};

\draw [-] (20)--(250);
\draw [dotted] (20)--(350);
\draw [-] (30)--(350);
\end{tikzpicture}
is of class $W$, then so are
\begin{equation*}
\begin{tikzpicture}
\node[bullet] (10) at (1,0) [label=above:{$-2$}] {};
\node[bullet] (20) at (2,0) [label=above:{$-b_1$}] {};

\node[empty] (250) at (2.5,0) [] {};
\node[empty] (30) at (3,0) [] {};

\node[bullet] (350) at (3.5,0) [label=above:{$-b_{r-1}$}] {};
\node[bullet] (450) at (4.5,0) [label=above:{$-b_r-1$}] {};

\draw [-] (10)--(20);
\draw [-] (20)--(250);
\draw [dotted] (20)--(350);
\draw [-] (30)--(350);
\draw [-] (350)--(450);
\end{tikzpicture}
\text{~and~}
\begin{tikzpicture}
\node[bullet] (10) at (1,0) [label=above:{$-b_1-1$}] {};
\node[bullet] (20) at (2,0) [label=above:{$-b_2$}] {};

\node[empty] (250) at (2.5,0) [] {};
\node[empty] (30) at (3,0) [] {};

\node[bullet] (350) at (3.5,0) [label=above:{$-b_r$}] {};
\node[bullet] (450) at (4.5,0) [label=above:{$-2$}] {};

\draw [-] (10)--(20);
\draw [-] (20)--(250);
\draw [dotted] (20)--(350);
\draw [-] (30)--(350);
\draw [-] (350)--(450);
\end{tikzpicture}
\end{equation*}

\item Every linear chain of class $W$ can be obtained by starting with \begin{tikzpicture}
\node[bullet] (00) at (0,0) [label=above:{$-4$}] {};
\end{tikzpicture} and iterating the steps described in (2).
\end{enumerate}
\end{proposition}

\begin{observation}
The Riemenschneider's dot diagram for a linear chain of class $W$ is symmetric to a special dot in the diagram which will be defined in the proof.
\end{observation}

\begin{proof}
The Riemenschneider's diagram for \begin{tikzpicture}
\node[bullet] (00) at (0,0) [label=above:{$-4$}] {};
\end{tikzpicture} is
\begin{equation*}
\begin{tikzpicture}[scale=0.5]
\node[bullet] (00) at (0,0) [] {};
\node[bullet] (10) at (1,0) [label=above:$\delta$] {};
\node[bullet] (20) at (2,0) [] {};
\end{tikzpicture}
\end{equation*}
which is symmetric to the dot decorated by $\delta$. And the step (2) in Proposition~\ref{proposition:linear-chain-of-class-T} above can be described in the view of Riemenschneider's dot diagram as follows: The procedure
\begin{center}
from
\begin{tikzpicture}
\node[bullet] (20) at (2,0) [label=above:{$-b_1$}] {};

\node[empty] (250) at (2.5,0) [] {};
\node[empty] (30) at (3,0) [] {};

\node[bullet] (350) at (3.5,0) [label=above:{$-b_r$}] {};

\draw [-] (20)--(250);
\draw [dotted] (20)--(350);
\draw [-] (30)--(350);
\end{tikzpicture}
to
\begin{tikzpicture}
\node[bullet] (10) at (1,0) [label=above:{$-2$}] {};
\node[bullet] (20) at (2,0) [label=above:{$-b_1$}] {};

\node[empty] (250) at (2.5,0) [] {};
\node[empty] (30) at (3,0) [] {};

\node[bullet] (350) at (3.5,0) [label=above:{$-b_{r-1}$}] {};
\node[bullet] (450) at (4.5,0) [label=above:{$-b_r-1$}] {};

\draw [-] (10)--(20);
\draw [-] (20)--(250);
\draw [dotted] (20)--(350);
\draw [-] (30)--(350);
\draw [-] (350)--(450);
\end{tikzpicture}
\end{center}
can be understood as  adding a dot to the right of the last dot of the final row and adding a dot over the first dot of the first row in the given Riemenschneider's dot diagram. Also the procedure
\begin{center}
from
\begin{tikzpicture}
\node[bullet] (20) at (2,0) [label=above:{$-b_1$}] {};

\node[empty] (250) at (2.5,0) [] {};
\node[empty] (30) at (3,0) [] {};

\node[bullet] (350) at (3.5,0) [label=above:{$-b_r$}] {};

\draw [-] (20)--(250);
\draw [dotted] (20)--(350);
\draw [-] (30)--(350);
\end{tikzpicture}
to
\begin{tikzpicture}
\node[bullet] (10) at (1,0) [label=above:{$-b_1-1$}] {};
\node[bullet] (20) at (2,0) [label=above:{$-b_2$}] {};

\node[empty] (250) at (2.5,0) [] {};
\node[empty] (30) at (3,0) [] {};

\node[bullet] (350) at (3.5,0) [label=above:{$-b_r$}] {};
\node[bullet] (450) at (4.5,0) [label=above:{$-2$}] {};

\draw [-] (10)--(20);
\draw [-] (20)--(250);
\draw [dotted] (20)--(350);
\draw [-] (30)--(350);
\draw [-] (350)--(450);
\end{tikzpicture}
\end{center}
is just equivalent to adding a dot to the left of the first dot of the top row and a dot under the last dot of the final row. Therefore the Riemenschneider's diagram of a linear chain of class $W$ is still symmetric to the dot decorated by $\delta$.
\end{proof}

\begin{example}[Continued from Example~\ref{example:cyclic-49/34}]
The Riemenschneider's dot diagram for $49/34$ is symmetric to the dot decorated by $\delta$:
\begin{equation*}
\begin{tikzpicture}[scale=0.5]
\node[bullet] at (0,3) [label=above:{}] {};

\node[bullet] at (0,2) [] {};

\node[bullet] at (0,1) [] {};
\node[bullet] at (1,1) [] {};
\node[bullet] at (2,1) [label=above:{$\delta$}] {};
\node[bullet] at (3,1) [] {};

\node[bullet] at (3,0) [] {};
\node[bullet] at (4,0) [] {};
\node[bullet] at (5,0) [] {};

\draw [-] (0,3)--(0,2)--(0,1)--(1,1);
\draw [-] (3,1)--(3,0)--(4,0)--(5,0);
\end{tikzpicture}
\end{equation*}
\end{example}

\begin{definition}\label{definition:delta-half-linear-chain}
The \emph{$\delta$-position $(i(\delta),j(\delta))$} of a linear chain
\begin{tikzpicture}
\node[bullet] (20) at (2,0) [label=above:{$-b_1$}] {};

\node[empty] (250) at (2.5,0) [] {};
\node[empty] (30) at (3,0) [] {};

\node[bullet] (350) at (3.5,0) [label=above:{$-b_r$}] {};
\draw [-] (20)--(250);
\draw [dotted] (20)--(350);
\draw [-] (30)--(350);
\end{tikzpicture}
of class $W$ is the $i(\delta)$-th row and the $j(\delta)$-th column containing the center of the symmetry. The \emph{$\delta$-half linear chain} of a linear chain of class $W$ with the $\delta$-position $(i(\delta), j(\delta))$ is a linear chain
\begin{equation*}
\begin{tikzpicture}
\node[bullet] (20) at (2,0) [label=above:{$-b_1$}] {};

\node[empty] (250) at (2.5,0) [] {};
\node[empty] (30) at (3,0) [] {};

\node[bullet] (350) at (3.5,0) [label=above:{$-b_{i(\delta)-1}$}] {};
\node[bullet] (450) at (4.5,0) [label=above:{$-b_{i(\delta)}'$}] {};
\draw [-] (20)--(250);
\draw [dotted] (20)--(350)--(450);
\draw [-] (350)--(450);
\end{tikzpicture}
\end{equation*}
whose Riemenschneider's dot diagram is obtained by choosing from the first row to the $i(\delta)$-th row and from the first column to the $j(\delta)$-th column of the original Riemenschneider's dot diagram.
\end{definition}

\begin{example}[Continued from Example~\ref{example:cyclic-49/34}]
\label{example:cyclic-49/34-delta-half}
The $\delta$-position of the linear chain corresponding to $49/34$ is $(3,3)$. So the $\delta$-half linear chain corresponding to $49/34$ is
\begin{equation*}
\begin{tikzpicture}
\node[bullet] (00) at (0,0) [label=above:$-2$] {};
\node[bullet] (10) at (1,0) [label=above:$-2$] {};
\node[bullet] (20) at (2,0) [label=above:$-4$] {};

\draw [-] (00)--(10)--(20);
\end{tikzpicture}
\end{equation*}
whose Riemenschneider's dot diagram is obtained as follows:
\begin{equation*}
\begin{tikzpicture}[scale=0.5]
\node[bullet] at (0,3) [label=above:{}] {};

\node[bullet] at (0,2) [] {};

\node[bullet] at (0,1) [] {};
\node[bullet] at (1,1) [] {};
\node[bullet] at (2,1) [label=above:{$\delta$}] {};
\node[bullet] at (3,1) [] {};

\node[bullet] at (3,0) [] {};
\node[bullet] at (4,0) [] {};
\node[bullet] at (5,0) [] {};

\draw (-0.25,0.75) rectangle (2.25,3.25);
\end{tikzpicture}
\end{equation*}
\end{example}

\subsection{Rational homology balls corresponding to linear chains of class $W$}
\label{subsection:classT}

As mentioned in Introduction, the boundary of the $4$-manifold $C_{p,q}$ obtained by plumbing the linear chain of class $W$ corresponding to $p^2/(pq-1)$ with $1 \le q < p$ is the Lens space $L(p^2,pq-1)$, which also bounds a rational homology ball $B_{p,q}$ due to Casson--Harer~\cite{Casson-Harer-1981}.

On the other hand, one may interpret $B_{p,q}$ as a Milnor fiber of a cyclic quotient singularity of type $\frac{1}{p^2}(1,pq-1)$. We briefly recall some relevant notions and results. A \emph{cyclic quotient singularity of type $\frac{1}{n}(1,a)$} for $1 \le a < n$ with $(n,a)=1$ is a quotient surface singularity $\mathbb{C}^2/\mu_n$, where $\mu_n$ is a cyclic multiplicative group generated by a $n$-th root of unity $\zeta$ acting on $\mathbb{C}^2$ by $\zeta \cdot (x,y) = (\zeta x, \zeta^a y)$. Then the dual graph of the minimal resolution of a cyclic quotient singularity of type $\frac{1}{n}(1,a)$ is that of the linear chain corresponding to $n/a$. A \emph{smoothing} $\pi \colon (\mathcal{X},0) \to \Delta$ of a normal surface singularity $(X,0)$ is a proper flat map from a 3-dimensional isolated singularity $(\mathcal{X},0)$ to a small disk $\Delta = \{t \in \mathbb{C} : \abs{t} < \epsilon\}$ such that $(\pi^{-1}(0),0) \cong (X,0)$ and $\pi^{-1}(t)$ is smooth for all $t \neq 0$. The \emph{Milnor fiber} $M$ of a smoothing $\pi \colon (\mathcal{X}, 0) \to \Delta$ is roughly speaking a nearby fiber of the central fiber $\pi^{-1}(0)$. According to a general theory of Milnor fibers, $M$ is a compact $4$-manifold whose boundary is the link $L$ of the singularity $(X,0)$, where the \emph{link} $L = X \cap \partial B$ of $(X,0)$ is defined by a boundary of a small neighborhood $B$ of the singularity.

In case of a cyclic quotient singularity of type $\frac{1}{n}(1,a)$, the link $L$ is the lens space $L(n,a)$. Furthermore, for a cyclic quotient singularity of type $\frac{1}{p^2}(1,pq-1)$ with $1 \le q < p$, there is a smoothing $\pi \colon (\mathcal{X},0) \to \Delta$ (so-called \emph{$\mathbb{Q}$-Gorenstein smoothing}) whose Milnor fiber is a rational homology ball $B_{p,q}$.

\begin{remark}\label{remark:Q-bldn=smoothing}
Suppose that $\widetilde{X}$ is a smooth complex surface containing the linear chain of complex rational curves whose dual graph is given as in \eqref{equation:Cpq}. Then, since $\widetilde{X}$ contains a configuration $C_{p,q}$, one can perform a rational blow-down surgery along $C_{p,q}$ to obtain a new smooth $4$-manifold, say $Z$.
On the other hand, let $X$ be a singular complex surface with a singularity $o \in X$ obtained by contracting the linear chain \eqref{equation:Cpq} to $o$. Then one may interpret a rational blow-down surgery on $\widetilde{X}$ as a global smoothing of $X$ if there is no local-to-global obstruction to deformation. That is, if $\pi \colon \smash[b]{\mathcal{X}} \to \Delta$ is a smoothing of $X$ induced by a local smoothing of $o \in X$, then its general fiber $\pi^{-1}(t)$ ($t \neq 0$) is diffeomorphic to $Z$.
\end{remark}

\begin{definition}
A \emph{Wahl singularity} is a cyclic quotient singularity which admits a smoothing whose Milnor fiber is a rational homology ball.
\end{definition}

\begin{proposition}[cf.~Koll\'ar--Shepherd-Barron~\cite{Kollar-Shepherd-Barron-1988}]
A cyclic quotient singularity of type $\frac{1}{n}(1,a)$ is a Wahl singularity if and only if $n=p^2$ and $a=pq-1$ for some $1 \le q < p$.
\end{proposition}

\begin{corollary}
The dual graph of the minimal resolution of a Wahl singularity is a linear chain of class $W$ given in Proposition~\ref{proposition:linear-chain-of-class-T}.
\end{corollary}

\section{Flips in minimal model program}
\label{section:flip}

We review some basics of divisorial contractions and flips in minimal model program from Koll{\'a}r--Mori~\cite{Kollar-Mori-1992} and HTU~\cite{Hacking-Tevelev-Urzua-2013}

\begin{definition}
A three dimensional \emph{extremal neighborhood} is a proper birational morphism $f \colon (C \subset \mathcal{W}) \to (Q \in \mathcal{Z})$ satisfying the following properties:
\begin{enumerate}
\item The canonical class $K_{\mathcal{W}}$ is $\mathbb{Q}$-Cartier and $\mathcal{W}$ has only terminal singularities;

\item $\mathcal{Z}$ is normal with a distinguished point $Q \in \mathcal{Z}$;

\item $C=f^{-1}(Q)$ is an irreducible curve;

\item $K_{\mathcal{W}} \cdot C < 0$.
\end{enumerate}
If the exceptional set of $f$ is $C$, then the extremal neighborhood is said to be \emph{flipping}. On the other hand, if it is not flipping, then the exceptional set of $f$ is of dimension $2$. In this case we call it by a \emph{divisorial} extremal neighborhood.
\end{definition}

An extremal neighborhood that we are concerned with is the following:

\begin{definition}
Let $f \colon W \to Z$ be a partial resolution of a two dimensional cyclic quotient singularity germ $(Q \in Z)$ such that $f^{-1}(Q)=C$ is a smooth rational curve with one Wahl singularity of $W$ on $C$. Suppose that $K_W \cdot C < 0$. Let $\mathcal{W} \to \Delta$ be a $\mathbb{Q}$-Gorenstein smoothing of $W$ and let $\mathcal{Z} \to \Delta$ be the corresponding blown-down deformation of $Z$; Wahl~\cite{Wahl-1976}. The induced birational morphism $(C \subset \mathcal{W}) \to (Q \in \mathcal{Z})$ is called an \emph{extremal neighborhood of type mk1A}.
\end{definition}

\begin{example}\label{example:flipping-extremal-nbhd}
Let $p^2/(pq-1)=[b_1,\dotsc,b_r]$ for $1 \le q < p$. Let $U$ be a regular neighborhood of a linear chain of $2$-spheres whose dual graph is given as follows:
\begin{equation*}
\begin{tikzpicture}
\node[bullet] (10) at (1,0) [label=above:{$-b_1$}] {};
\node[bullet] (20) at (2,0) [label=above:{$-b_2$}] {};

\node[empty] (250) at (2.5,0) [] {};
\node[empty] (30) at (3,0) [] {};

\node[bullet] (350) at (3.5,0) [label=above:{$-b_{r-1}$}] {};
\node[bullet] (450) at (4.5,0) [label=above:{$-b_r$}] {};
\node[bullet] (550) at (5.5,0) [label=above:{$-1$},label=below:{$C$}] {};

\draw [-] (10)--(20);
\draw [-] (20)--(250);
\draw [dotted] (20)--(350);
\draw [-] (30)--(350);
\draw [-] (350)--(450)--(550);
\end{tikzpicture}
\end{equation*}
Let $U \to W$ be the contraction of the linear chain corresponding to $p^2/(pq-1)$ in $U$ to a Wahl singularity $Q' \in W$ of type $\frac{1}{p^2}(1,pq-1)$. Denote again by $C \subset W$ the image of $C \subset U$. Let $U \to Z$ be the contraction of the whole above linear chain to a quotient singularity $Q \in Z$. Let $\mathcal{W} \to \Delta$ be a deformation of $W$ induced by the $\mathbb{Q}$-Gorenstein smoothing of the singularity $Q'$ and let $\mathcal{Z}$ be the blow-down deformation of $\mathcal{W} \to \Delta$ (cf.~Wahl~\cite{Wahl-1976}). Then it is not difficult to show that $(C \subset \mathcal{W}) \to (Q \in \mathcal{Z})$ is a flipping extremal neighborhood of type mk1A; cf.~HTU~\cite{Hacking-Tevelev-Urzua-2013}.
\end{example}

At first the divisorial contraction for an extremal neighborhood of type mk1A is just a blowing down of $(-1)$-curves:

\begin{proposition}[cf.~Urz{\'u}a~{\cite[Proposition~2.8]{Urzua-2013}}]
If $(C \subset \mathcal{W}) \to (Q \in \mathcal{Z})$ is a divisorial extremal neighborhood of type mk1A, then $(Q \in \mathcal{Z})$ is a Wahl singularity. The divisorial contraction $\mathcal{W} \to \mathcal{Z}$ induces the blowing down of $(-1)$-curves between the smooth fibers of $\mathcal{W} \to \Delta$ and $\mathcal{Z} \to \Delta$.
\end{proposition}

In case of a flipping extremal neighborhood, $K_{\mathcal{Z}}$ is not $\mathbb{Q}$-Cartier in general. So we modify it:

\begin{definition}
The \emph{flip} of a flipping extremal neighborhood $f \colon (C \subset \mathcal{W}) \to (Q \in \mathcal{Z})$  (or, if no confusion, the flip of $\mathcal{W}$) is a proper birational morphism $f^+ \colon (C^+ \subset \mathcal{W}^+) \to (Q \in \mathcal{Z})$ where $\mathcal{W}^+$ is normal with only terminal singularities such that the exceptional set of $f^+$ is $C^+$ and $K_{\mathcal{W}^+}$ is $\mathbb{Q}$-Cartier and $f^+$-ample.
\end{definition}

The flipped surface $(C^+ \subset W^+)$ can be computed by a special partial resolution of $(Q \in Z)$.

\begin{definition}[HTU~\cite{Hacking-Tevelev-Urzua-2013}]
An \emph{extremal $P$-resolution} of a two dimensional cyclic quotient singularity germ $(Q \in Z)$ is a partial resolution $f^+ \colon W^+ \to Z$ such that $C^+=(f^+)^{-1}(Q)$ is a smooth rational curve and  $W^+$ has only at most two Wahl singularities and $K_{W^+}$ is ample relative to $f^+$.
\end{definition}

\begin{proposition}[Koll{\'a}r--Mori~\cite{Kollar-Mori-1992}]
\label{proposition:Kollar-Mori}
Suppose that $(C \subset \mathcal{W}) \to (Q \in \mathcal{Z})$ is a flipping extremal neighborhood of type mk1A. Let $(C \subset W) \to (Q \in Z)$ be the contraction of $C$ between the central fibers $W$ and $Z$. Then there exists an extremal $P$-resolution $(C^+ \subset W^+) \to (Q \in Z)$ such that the flip $(C^+ \subset \mathcal{W}^+) \to (Q \in \mathcal{Z})$ is obtained by the blown-down deformation of a $\mathbb{Q}$-Gorenstein smoothing of $W^+$. That is, we have the commutative diagram
\begin{center}
\includegraphics{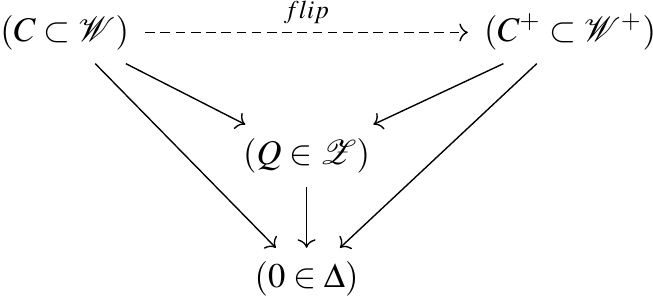}
\end{center}
%
which is restricted to the central fibers as follows:
\begin{center}
\includegraphics{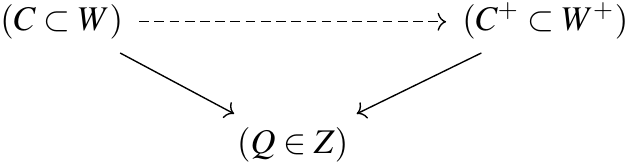}
\end{center}
\end{proposition}

In HTU~\cite{Hacking-Tevelev-Urzua-2013}, they described explicitly the numerical data of the central fibers $(C \subset W) \to (Q \in Z)$ and $(C^+ \subset W^+)$ of an extremal neighborhood of type mk1A. For details, refer HTU~\cite{Hacking-Tevelev-Urzua-2013}.

At first, the data for $(C \subset W) \to (Q \in Z)$ is given as follows: Let $(C \subset W)$ be the central fiber of an extremal neighborhood of type mk1A with a Wahl singularity of type $\frac{1}{p^2}(1, pq-1)$ lying on $C$. Let
\begin{equation*}
\frac{p^2}{pq-1} = [b_1, \dotsc, b_r]
\end{equation*}
and let $E_1, \dotsc, E_r$ be the exceptional curves of the minimal resolution $\widetilde{W}$ of $W$ with $E_j \cdot E_j=-b_j$ for all $j$. Since $K_W \cdot C < 0$ and $C \cdot C < 0$, the strict transform of $C$, denoted again by $C$, is a $(-1)$-curve intersecting only one exceptional curve, say $E_i$, at one point. We denote this data by
\begin{equation*}
(b_1, \dotsc, \underline{b}_i, \dotsc, b_r).
\end{equation*}
Then $(Q \in Z)$ is a cyclic quotient singularity of type $\frac{1}{\Delta}(1, \Omega)$ where
\begin{equation}\label{equation:flip-in-the-middle}
\frac{\Delta}{\Omega} = [b_1, \dotsc, b_i-1, \dotsc, b_r].
\end{equation}

On the other hand, the birational map $(C^+ \subset W^+) \to (Q \in Z)$ for a flipping extremal neighborhood of type mk1A which appears frequently in calculation in this paper is described as follows:

\begin{proposition}[HTU~\cite{Hacking-Tevelev-Urzua-2013}]
\label{proposition:flip}
Let $(b_1, \dotsc, b_{r-1}, \underline{b}_r)$ be a data of a flipping extremal neighborhood of type mk1A.
Suppose that $i$ is the largest index satisfying $b_i \ge 3$ and $b_j=2$ for all $i < j \le  r$. Then the image of $E_1$ in $W^+$ is the curve $C^+$ with the Wahl singularity of type $\frac{1}{m^2}(1, ma-1)$ where $\frac{m^2}{ma-1}=[b_2, \dotsc, b_i-1]$.
\end{proposition}

\begin{example}[Continued from Example~\ref{example:cyclic-49/34}]
\label{example:cyclic-49/34-flipped}
Consider a flipping extremal neighborhood of type mk1A whose data is given by $[2,2,5,\underbar{4}]$, that is,
\begin{equation*}
\begin{tikzpicture}
\node[rectangle] (00) at (0,0) [label=above:$-2$] {};
\node[rectangle] (10) at (1,0) [label=above:$-2$] {};
\node[rectangle] (20) at (2,0) [label=above:$-5$] {};
\node[rectangle] (30) at (3,0) [label=above:$-4$] {};
\node[bullet] (40) at (4,0) [label=above:$-1$,label=below:{$C$}] {};

\draw [-] (00)--(10)--(20)--(30)--(40);
\end{tikzpicture}
\end{equation*}
where the linear chain of the vertices $\square$ is contracted to the Wahl singularity of type $\frac{1}{49}(1,34)$ on $(C \subset W)$. According to Proposition~\ref{proposition:flip} above, the Hirzebruch--Jung continued fraction for the Wahl singularity on $(C^+ \subset W^+)$ is $[2,5,3]$, which can be represented as follows:
\begin{equation*}
\begin{tikzpicture}
\node[bullet] (00) at (0,0) [label=above:$-2$,label=below:{$C^+$}] {};
\node[rectangle] (10) at (1,0) [label=above:$-2$] {};
\node[rectangle] (20) at (2,0) [label=above:$-5$] {};
\node[rectangle] (30) at (3,0) [label=above:$-3$] {};

\draw [-] (00)--(10)--(20)--(30);
\end{tikzpicture}
\end{equation*}
\end{example}

One can interpret the above flip of a flipping extremal neighborhood of type mk1A whose data is $(b_1, \dotsc, b_{r-1}, \underline{b}_r)$ via Riemenschneider's dot diagram as follows:

\begin{observation}\label{observation:flip-via-dot-diagram}
From the Riemenschneider's dot diagram associated to $[b_1,\dotsc,b_r]$, we first delete its last column so that we get a modified Riemenschneider's dot diagram and we then keep the first column of the modified Riemenschneider's dot diagram for representing $C^+$. Then the remained part is the Riemenschneider's dot diagram for $[b_2,\dotsc,b_i-1]$, which corresponds to the Wahl singularity on $(C^+ \subset W^+)$.
\end{observation}

\begin{example}[Continued from Example~\ref{example:cyclic-49/34-flipped}]
\label{example:cyclic-49/34-flipped-Riemenschneider}
The Riemenschneider's dot diagram of the Wahl singularity on $(C^+ \subset W^+)$ is the following:
\begin{equation*}
\begin{tikzpicture}[scale=0.5]
\node[bullet] at (0,3) [label=above:{}] {};

\node[bullet] at (0,2) [] {};

\node[bullet] at (0,1) [] {};
\node[bullet] at (1,1) [] {};
\node[bullet] at (2,1) [label=above:{$\delta$}] {};
\node[bullet] at (3,1) [] {};

\node[bullet] at (3,0) [] {};
\node[bullet] at (4,0) [] {};
\node[bullet] at (5,0) [] {};

\draw [-] (5,-0.5)--(5,0.5);
\draw (-0.25,2.75) rectangle (0.25,3.25);
\draw (-0.25,-0.25) rectangle (4.25,2.25);

\node at (-1,3) {$C^+$};
\end{tikzpicture}
\end{equation*}

\end{example}

\subsection{Flips and $\delta$-half linear chains}
\label{subsection:flip->delta-half}

Suppose that $p^2/(pq-1)=[b_1,\dotsc,b_r]$ and its dual $p^2/(p^2-pq+1)=[a_1,\dotsc,a_e]$ for $1 \le q < p$. Let $(i(\delta), j(\delta))$ be the $\delta$-position of the linear chain corresponding to $p^2/(pq-1)$. Let $U$ be a regular neighborhood of a linear chain of $2$-spheres whose dual graph is given as follows:
\begin{equation}\label{equation:linear-chain-p^2/(pq-1)}
\begin{tikzpicture}
\node[bullet] (10) at (1,0) [label=above:{$-b_1$},label=below:{$B_1$}] {};
\node[bullet] (20) at (2,0) [label=above:{$-b_2$},label=below:{$B_2$}] {};

\node[empty] (250) at (2.5,0) [] {};
\node[empty] (30) at (3,0) [] {};

\node[bullet] (350) at (3.5,0) [label=above:{$-b_{r-1}$},label=below:{$B_{r-1}$}] {};
\node[bullet] (450) at (4.5,0) [label=above:{$-b_r$},label=below:{$B_r$}] {};

\node[bullet] (550) at (5.5,0) [label=above:{$-1$},label=below:{$C$}] {};

\node[bullet] (650) at (6.5,0) [label=above:{$-a_e$},label=below:{$A_e$}] {};
\node[empty] (70) at (7,0) [] {};
\node[empty] (750) at (7.5,0) [] {};
\node[bullet] (80) at (8,0) [label=above:{$-a_{j(\delta)+2}$},label=below:{$A_{j(\delta)+2}$}] {};

\draw [-] (10)--(20);
\draw [-] (20)--(250);
\draw [dotted] (20)--(350);
\draw [-] (30)--(350);
\draw [-] (350)--(450);
\draw [-] (450)--(550);

\draw [-] (550)--(650);
\draw [-] (650)--(70);
\draw [dotted] (650)--(80);
\draw [-] (750)--(80);
\end{tikzpicture}
\end{equation}
Let $U \to W$ be the contraction of the linear chain corresponding to $p^2/(pq-1)$ in $U$ to a Wahl singularity $Q' \in W$ of type $\frac{1}{p^2}(1,pq-1)$. Denote again by $C \subset W$ the image of $C \subset U$. Let $U \to Z$ be the contraction of the linear chain
\begin{equation*}
\begin{tikzpicture}
\node[bullet] (10) at (1,0) [label=above:{$-b_1$},label=below:{$B_1$}] {};
\node[bullet] (20) at (2,0) [label=above:{$-b_2$},label=below:{$B_2$}] {};

\node[empty] (250) at (2.5,0) [] {};
\node[empty] (30) at (3,0) [] {};

\node[bullet] (350) at (3.5,0) [label=above:{$-b_{r-1}$},label=below:{$B_{r-1}$}] {};
\node[bullet] (450) at (4.5,0) [label=above:{$-b_r$},label=below:{$B_r$}] {};
\node[bullet] (550) at (5.5,0) [label=above:{$-1$},label=below:{$C$}] {};

\draw [-] (10)--(20);
\draw [-] (20)--(250);
\draw [dotted] (20)--(350);
\draw [-] (30)--(350);
\draw [-] (350)--(450)--(550);
\end{tikzpicture}
\end{equation*}
to a quotient singularity $Q \in Z$. Let $\mathcal{W} \to \Delta$ be a deformation of $W$ induced by the $\mathbb{Q}$-Gorenstein smoothing of the singularity $Q'$ and let $\mathcal{Z}$ be the blown-down deformation of $\mathcal{W} \to \Delta$. Then $(C \subset \mathcal{W}) \to (Q \in \mathcal{Z})$ is a flipping extremal neighborhood of type mk1A.

\begin{corollary}\label{corollary:flip->delta-half}
Applying flips described in Proposition~\ref{proposition:flip} repeatedly to $(C \subset \mathcal{W}) \to (Q \in \mathcal{Z})$, we have a deformation $\mathcal{Y} \to \Delta$ such that all of the fibers are smooth. In particular, the central fiber $Y_0$ of $\mathcal{Y} \to \Delta$ is a regular neighborhood of the $\delta$-half linear chain associated to $p^2/(pq-1)$.
\end{corollary}

\begin{proof}
Suppose that $i$ is the largest index satisfying $b_i \ge 3$ and $b_j=2$ for all $i < j \le  r$. Then, by Riemenschneider's dot diagram, we have $a_e=r-i+2$. If we apply a flip once to $(C \subset \mathcal{W})$, then we have to blow down $(r-i+1)$-times the $(-1)$-curves starting from $C \subset U$. Then all curves $B_j$ for $j > i$ are killed during the blowing-downs but only $A_e$ in the dual part is transformed into the new $(-1)$-curve $A_e'$, while we keep $B_1$ for $C^+$. Then the resulting linear chain is as follows:
\begin{equation}\label{equation:resulting-linear-chain}
\begin{tikzpicture}
\node[bullet] (10) at (1,0) [label=above:{$-b_1$},label=below:{$C^+$}] {};
\node[rectangle] (20) at (2,0) [label=above:{$-b_2$},label=below:{$B_2$}] {};

\node[empty] (250) at (2.5,0) [] {};
\node[empty] (30) at (3,0) [] {};

\node[rectangle] (350) at (3.5,0) [label=above:{$-b_{i-1}$},label=below:{$B_{i-1}$}] {};
\node[rectangle] (450) at (4.5,0) [label=above:{$-b_i+1$},label=below:{$B_i'$}] {};

\node[bullet] (550) at (5.5,0) [label=above:{$-1$},label=below:{$A_e'$}] {};

\node[bullet] (650) at (6.5,0) [label=above:{$-a_{e-1}$},label=below:{$A_{e-1}$}] {};
\node[empty] (70) at (7,0) [] {};
\node[empty] (750) at (7.5,0) [] {};
\node[bullet] (80) at (8,0) [label=above:{$-a_{j(\delta)+2}$},label=below:{$A_{j(\delta)+2}$}] {};

\draw [-] (10)--(20);
\draw [-] (20)--(250);
\draw [dotted] (20)--(350);
\draw [-] (30)--(350);
\draw [-] (350)--(450);
\draw [-] (450)--(550);

\draw [-] (550)--(650);
\draw [-] (650)--(70);
\draw [dotted] (650)--(80);
\draw [-] (750)--(80);
\end{tikzpicture}
\end{equation}
where the linear chain of the rectangles $\square$ is contracted to a new Wahl singularity on $C^+$.

The new $(-1)$-curve $A_e'$ is again the flipping curve. So we can continue to flip. But this process can be continued only until the $(j(\delta)+2)$-th column is deleted because of the symmetry of Riemenschneider's dot diagram. Then after the final flip, the remained linear chain is the following:
\begin{equation*}
\begin{tikzpicture}
\node[bullet] (10) at (1,0) [label=above:{$-b_1$},label=below:{$B_1$}] {};
\node[bullet] (20) at (2,0) [label=above:{$-b_2$},label=below:{$B_2$}] {};

\node[empty] (250) at (2.5,0) [] {};
\node[empty] (30) at (3,0) [] {};

\node[bullet] (350) at (3.5,0) [label=above:{$-b_{i(\delta)-1}$},label=below:{$B_{i(\delta)-1}$}] {};
\node[bullet] (450) at (4.5,0) [label=above:{$-b_{i(\delta)}+1$},label=below:{$B_{i(\delta)}'$}] {};

\draw [-] (10)--(20);
\draw [-] (20)--(250);
\draw [dotted] (20)--(350);
\draw [-] (30)--(350);
\draw [-] (350)--(450);
\end{tikzpicture}
\end{equation*}
Therefore we get a deformation $\mathcal{Y} \to \Delta$ such that its general fiber contains the $\delta$-half linear chain associated to $p^2/(pq-1)$. Furthermore, in the view of $U$, a flip is just blow-downs of  $(-1)$-curves. Therefore, in each step, the minimal resolution of the central fiber of the flipped deformation $(C^+ \subset \mathcal{W}^+)$ is a regular neighborhood of the resulting linear chain in \eqref{equation:resulting-linear-chain}. Hence the central fiber $Y_0$ is just a regular neighborhood of the $\delta$-half linear chain.
\end{proof}

\begin{corollary}\label{corollary:delta-half->chain}
By blowing up appropriately the $\delta$-linear chain corresponding to $p^2/(pq-1)$ with $1 \le q < p$, we obtain the linear chain in \eqref{equation:linear-chain-p^2/(pq-1)}.
\end{corollary}

\begin{proof}
In the level of minimal resolutions, the flips in the proof of Corollary~\ref{corollary:flip->delta-half} above is just a sequence of blowing-downs. So the proof above says that one can obtain the $\delta$-half linear chain from the linear chain in \eqref{equation:linear-chain-p^2/(pq-1)} by blowing down appropriately; hence, the assertion follows.
\end{proof}

\begin{example}[Continued from Examples~\ref{example:cyclic-49/34-flipped} and \ref{example:cyclic-49/34-flipped-Riemenschneider}]
\label{example:cyclic-49/34-final}
We apply Corollary~\ref{corollary:flip->delta-half} above to the linear chain associated to $49/34$ as in Figure~\ref{figure:Example-49/34-flips}. Then we get a deformation $\mathcal{Y} \to \Delta$ such that all of its fibers are smooth and the central fiber is a regular neighborhood of the $\delta$-half linear chain
\begin{tikzpicture}[scale=0.5]
\node[bullet] (00) at (0,0) [label=above:$-2$] {};
\node[bullet] (10) at (1,0) [label=above:$-2$] {};
\node[bullet] (20) at (2,0) [label=above:$-4$] {};

\draw [-] (00)--(10)--(20);
\end{tikzpicture}.
In the view of Riemenschneider's dot diagram, the sequence of flips can be visualized as in Figure~\ref{figure:Example-49/34-Riemenschneider}.

\begin{figure}
\centering
\begin{tikzpicture}
\node[rectangle] (00) at (0,0) [label=above:$-2$] {};
\node[rectangle] (10) at (1,0) [label=above:$-2$] {};
\node[rectangle] (20) at (2,0) [label=above:$-5$] {};
\node[rectangle] (30) at (3,0) [label=above:$-4$] {};
\node[bullet] (40) at (4,0) [label=above:$-1$,label=below:{$C_1$}] {};
\node[bullet] (50) at (5,0) [label=above:$-2$] {};
\node[bullet] (60) at (6,0) [label=above:$-2$] {};

\draw [-] (00)--(10)--(20)--(30)--(40)--(50)--(60);
\end{tikzpicture}

$\downarrow$ \medskip

\begin{tikzpicture}
\node[bullet] (00) at (0,0) [label=above:$-2$,label=below:{$C_1^+$}] {};
\node[rectangle] (10) at (1,0) [label=above:$-2$] {};
\node[rectangle] (20) at (2,0) [label=above:$-5$] {};
\node[rectangle] (30) at (3,0) [label=above:$-3$] {};
\node[bullet] (50) at (5,0) [label=above:$-1$,label=below:{$C_2$}] {};
\node[bullet] (60) at (6,0) [label=above:$-2$] {};

\draw [-] (00)--(10)--(20)--(30)--(50)--(60);
\end{tikzpicture}

$\downarrow$ \medskip

\begin{tikzpicture}
\node[bullet] (00) at (0,0) [label=above:$-2$,label=below:{$C_1^+$}] {};
\node[bullet] (10) at (1,0) [label=above:$-2$,label=below:{$C_2^+$}] {};
\node[rectangle] (20) at (2,0) [label=above:$-5$] {};
\node[rectangle] (30) at (3,0) [label=above:$-2$] {};
\node[bullet] (60) at (6,0) [label=above:$-1$,label=below:{$C_3$}] {};

\draw [-] (00)--(10)--(20)--(30)--(60);
\end{tikzpicture}

$\downarrow$ \medskip

\begin{tikzpicture}
\node[bullet] (00) at (0,0) [label=above:$-2$,label=below:{$C_1^+$}] {};
\node[bullet] (10) at (1,0) [label=above:$-2$,label=below:{$C_2^+$}] {};
\node[bullet] (20) at (2,0) [label=above:$-4$,label=below:{$C_3^+$}] {};
\node (60) at (6,0) [label=above:{\phantom{$-1$}}] {};

\draw [-] (00)--(10)--(20);
\end{tikzpicture}

\caption{A sequence of flips for Example~\ref{example:cyclic-49/34-final}}
\label{figure:Example-49/34-flips}
\end{figure}
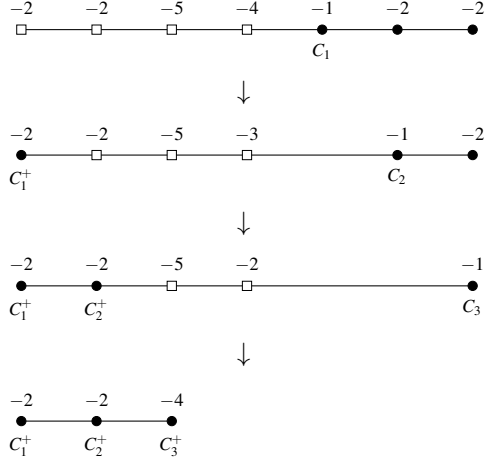

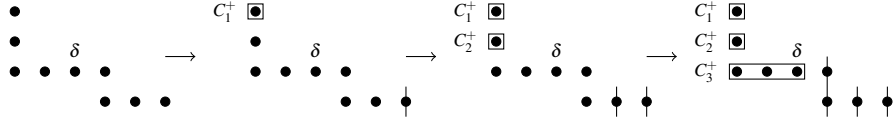
\begin{figure}
\begin{tikzpicture}[scale=0.4]

\node[bullet] at (0,3) [label=above:{}] {};

\node[bullet] at (0,2) [] {};

\node[bullet] at (0,1) [] {};
\node[bullet] at (1,1) [] {};
\node[bullet] at (2,1) [label=above:{$\delta$}] {};
\node[bullet] at (3,1) [] {};

\node[bullet] at (3,0) [] {};
\node[bullet] at (4,0) [] {};
\node[bullet] at (5,0) [] {};

\draw [->] (5,1.5)--(6,1.5);


\node[bullet] at (8,3) [label=above:{}] {};

\node[bullet] at (8,2) [] {};

\node[bullet] at (8,1) [] {};
\node[bullet] at (9,1) [] {};
\node[bullet] at (10,1) [label=above:{$\delta$}] {};
\node[bullet] at (11,1) [] {};

\node[bullet] at (11,0) [] {};
\node[bullet] at (12,0) [] {};
\node[bullet] at (13,0) [] {};

\draw [-] (13,-0.5)--(13,0.5);

\draw (7.75,2.75) rectangle (8.25,3.25);
\node at (7,3) {$C_1^+$};

\draw [->] (13,1.5)--(14,1.5);


\node[bullet] at (16,3) [label=above:{}] {};

\node[bullet] at (16,2) [] {};

\node[bullet] at (16,1) [] {};
\node[bullet] at (17,1) [] {};
\node[bullet] at (18,1) [label=above:{$\delta$}] {};
\node[bullet] at (19,1) [] {};

\node[bullet] at (19,0) [] {};
\node[bullet] at (20,0) [] {};
\node[bullet] at (21,0) [] {};

\draw [-] (21,-0.5)--(21,0.5);
\draw [-] (20,-0.5)--(20,0.5);

\draw (15.75,2.75) rectangle (16.25,3.25);
\node at (15,3) {$C_1^+$};

\draw (15.75,1.75) rectangle (16.25,2.25);
\node at (15,2) {$C_2^+$};

\draw [->] (21,1.5)--(22,1.5);


\node[bullet] at (24,3) [label=above:{}] {};

\node[bullet] at (24,2) [] {};

\node[bullet] at (24,1) [] {};
\node[bullet] at (25,1) [] {};
\node[bullet] at (26,1) [label=above:{$\delta$}] {};
\node[bullet] at (27,1) [] {};

\node[bullet] at (27,0) [] {};
\node[bullet] at (28,0) [] {};
\node[bullet] at (29,0) [] {};

\draw [-] (29,-0.5)--(29,0.5);
\draw [-] (28,-0.5)--(28,0.5);
\draw [-] (27,-0.5)--(27,1.5);

\draw (23.75,2.75) rectangle (24.25,3.25);
\node at (23,3) {$C_1^+$};

\draw (23.75,1.75) rectangle (24.25,2.25);
\node at (23,2) {$C_2^+$};

\draw (23.75,0.75) rectangle (26.25,1.25);
\node at (23,1) {$C_3^+$};
\end{tikzpicture}

\caption{Riemenschneider's dot diagrams for Example~\ref{example:cyclic-49/34-final}}

\label{figure:Example-49/34-Riemenschneider}
\end{figure}
\end{example}

\section{Embedded rational homology balls}
\label{section:embedded-QHB}

We finally prove the existence of embedded rational homology balls.

\begin{theorem}\label{theorem:Bpq}
Let $V$ be a plumbing $4$-manifold of the $\delta$-half linear chain corresponding to $p^2/(pq-1)$ with $1 \le q < p$. Then there is an embedded rational homology ball $B_{p,q}$ in $V$.
\end{theorem}

\begin{proof}
Assume first that the $\delta$-half linear chain consists of complex rational curves $\mathbb{CP}^1$ and $V$ is a complex surface. According to Corollary~\ref{corollary:delta-half->chain}, by blowing up appropriately in $V$, we have a regular neighborhood $U$ of the linear chain
\begin{equation*}
\begin{tikzpicture}
\node[bullet] (10) at (1,0) [label=above:{$-b_1$},label=below:{$B_1$}] {};
\node[bullet] (20) at (2,0) [label=above:{$-b_2$},label=below:{$B_2$}] {};

\node[empty] (250) at (2.5,0) [] {};
\node[empty] (30) at (3,0) [] {};

\node[bullet] (350) at (3.5,0) [label=above:{$-b_{r-1}$},label=below:{$B_{r-1}$}] {};
\node[bullet] (450) at (4.5,0) [label=above:{$-b_r$},label=below:{$B_r$}] {};

\node[bullet] (550) at (5.5,0) [label=above:{$-1$},label=below:{$C$}] {};

\node[bullet] (650) at (6.5,0) [label=above:{$-a_e$},label=below:{$A_e$}] {};
\node[empty] (70) at (7,0) [] {};
\node[empty] (750) at (7.5,0) [] {};
\node[bullet] (80) at (8,0) [label=above:{$-a_{j(\delta)+2}$},label=below:{$A_{j(\delta)+2}$}] {};

\draw [-] (10)--(20);
\draw [-] (20)--(250);
\draw [dotted] (20)--(350);
\draw [-] (30)--(350);
\draw [-] (350)--(450);
\draw [-] (450)--(550);

\draw [-] (550)--(650);
\draw [-] (650)--(70);
\draw [dotted] (650)--(80);
\draw [-] (750)--(80);
\end{tikzpicture}
\end{equation*}
corresponding to $p^2/(pq-1)$ in \eqref{equation:linear-chain-p^2/(pq-1)}. Let $U \to W$ be the contraction of the linear chain corresponding to $p^2/(pq-1)$ to the Wahl singularity of type $\frac{1}{p^2}(1,pq-1)$. We denote again by $C \subset W$ the image $C \subset U$. Let $\mathcal{W} \to \Delta$ is the deformation of $W$ which is induced from the $\mathbb{Q}$-Gorenstein smoothing of the Wahl singularity of type $\frac{1}{p^2}(1,pq-1)$. Therefore there is an embedded rational homology ball $B_{p,q}$ in every general fiber $W_t$ ($t \neq 0$).

On the other hand, by Corollary~\ref{corollary:flip->delta-half} above, if applying flips appropriately, there is a deformation $\mathcal{Y} \to \Delta$ such that all of its fibers are smooth and the central fiber $Y_0$ is just the regular neighborhood $V$ of the $\delta$-half linear chain corresponding to $p^2/(pq-1)$. Furthermore, notice that a flip changes only the central fiber of $\mathcal{W} \to \Delta$. Therefore a general fiber $Y_t$ ($t \neq 0$) is isomorphic to $W_t$; hence $Y_t$ also contains a rational homology ball $B_{p,q}$. Furthermore, since every fiber is smooth, the deformation $\mathcal{Y} \to \Delta$ is locally trivial as a fibration of smooth differentiable 4-manifolds. So the central fiber $Y_0$ is diffeomorphic to a general fiber $Y_t$ ($0 < t \ll \epsilon$). Hence there is an embedded rational homology ball $B_{p,q}$ in the central fiber $Y_0=V$.

In general, suppose that a regular neighborhood $V$ contains the $\delta$-half linear chain consisting of smooth $2$-spheres. Then one can take a complex surface model $V_{\mathbb{C}}$ containing the $\delta$-half linear chain consisting of complex rational curves such that $V_{\mathbb{C}}$ is diffeomorphic to $V$. Then we can apply the argument above to $V_{\mathbb{C}}$.
\end{proof}

\begin{corollary}\label{corollary:main}
Suppose $Z$ is a smooth $4$-manifold which contains the $\delta$-half linear chain corresponding to $p^2/(pq-1)$ with $1 \le q < p$. Then there is a smoothly embedded rational homology ball $B_{p,q}$ in $Z$.
\end{corollary}

Notice that we can also obtain the main results in Khodorovskiy~\cite{Khodorovskiy-2014} as by-products.

\begin{corollary}[{Khodorovskiy~\cite[Theorem~1.2]{Khodorovskiy-2014}}]
\label{corollary:Bn}
Let $V_{-n-1}$ be a regular neighborhood of a smooth $2$-sphere with self-intersection number $-n-1$. Then there is an embedded rational homology ball $B_{n,1}$ in $V_{-n-1}$ for any $n \ge 2$.
\end{corollary}

\begin{proof}
The linear graph corresponding to $n^2/(n-1)$ is
\begin{equation*}
\begin{tikzpicture}
\node[bullet] (10) at (1,0) [label=above:{$-n-2$}] {};
\node[bullet] (20) at (2,0) [label=above:{$-2$}] {};

\node[empty] (250) at (2.5,0) [] {};
\node[empty] (30) at (3,0) [] {};

\node[bullet] (350) at (3.5,0) [label=above:{$-2$}] {};

\draw [-] (10)--(20);
\draw [-] (20)--(250);
\draw [dotted] (20)--(350);
\draw [-] (30)--(350);

\draw [thick, decoration={brace,mirror,raise=0.5em}, decorate] (20) -- (350)
node [pos=0.5,anchor=north,yshift=-0.75em] {$n-2$};
\end{tikzpicture}
\end{equation*}
Therefore the $\delta$-half linear chain corresponding to $n^2/(n-1)$ is
\begin{tikzpicture}
\node[bullet] (00) at (0,0) [label=above:$-n-1$] {};
\end{tikzpicture}
. Then the assertion follows from Theorem~\ref{theorem:Bpq}.
\end{proof}

\begin{proposition}[{Khodorovskiy~\cite[Theorem~1.3]{Khodorovskiy-2014}}]
\label{proposition:Bn-for-V4}
For any odd integer $n \ge 3$, there are embedded rational homology $4$-balls $B_{n,1}$ in a regular neighborhood $V_{-4}$ of a smooth $2$-sphere with self-intersection number $-4$. For any even integer $n \ge 3$, there is an embedding $B_{n,1} \hookrightarrow B_{2,1} \sharp \overline{\mathbb{CP}}^2$.
\end{proposition}

\begin{proof}
As in the proof of Theorem~\ref{theorem:Bpq} above, we may assume that the $(-4)$-curve is a complex rational curve in a complex surface $V_{-4}$. For any $n \ge 3$, by blowing up appropriately $V_{-4}$, we get a regular neighborhood $U$ of the following linear chain of $\mathbb{CP}^1$'s:
\begin{equation*}
\begin{tikzpicture}
\node[bullet] (10) at (1,0) [label=above:{$-n-2$}] {};
\node[bullet] (20) at (2,0) [label=above:{$-2$}] {};

\node[empty] (250) at (2.5,0) [] {};
\node[empty] (30) at (3,0) [] {};

\node[bullet] (350) at (3.5,0) [label=above:{$-2$}] {};

\node[bullet] (2-1) at (2,-1) [label=left:{$-1$},label=right:{$C$}] {};

\draw [-] (10)--(20);
\draw [-] (20)--(250);
\draw [dotted] (20)--(350);
\draw [-] (30)--(350);

\draw [-] (20)--(2-1);

\draw [thick, decoration={brace,mirror,raise=0.5em}, decorate] (20) -- (350) node [pos=0.5,anchor=north,yshift=-0.75em] {$n-2$};

\end{tikzpicture}
\end{equation*}
Now we contract the linear chain corresponding to $n^2/(n-1)$ in $U$ so that we get a singular surface $W$ with a cyclic quotient singularity $Q' \in C \subset W$ of type $\frac{1}{n^2}(1,n-1)$. Let $(C \subset W) \to (Q \in Z)$ be the contraction of $C \subset W$. Then $Q$ is a cyclic quotient singularity of type $\frac{1}{4}(1,1)$ because $[n+2,1,2,\dotsc,2]=[4]$. Let $\mathcal{W} \to \Delta$ be a deformation of $W$ induced by the $\mathbb{Q}$-Gorenstein smoothing of $Q'$ and let $\mathcal{Z} \to \Delta$ be the
blown-down deformation of $\mathcal{W} \to \Delta$. Notice that a general fiber $W_t$ ($t \neq 0$) contains a rational homology ball $B_{n,1}$.

Assume that $n \ge 3$ is an odd integer. According to Example~2.12 in Urz\'ua~\cite{Urzua-2013}, $(C \subset \mathcal{W}) \to (Q \in \mathcal{Z})$ is a flipping extremal neighborhood of type mk1A. So if we apply a flip to the curve $C \subset W$, then, by Equation~\eqref{equation:flip-in-the-middle}, the resulting deformation $\mathcal{W}^+ \to \Delta$ has a central fiber $W_0^+$ which is isomorphic to $V_{-4}$ because its extremal $P$-resolution is just the minimal resolution of $Q$ without any singularities. On the other hand, since the general fiber $W^+_t=W_t$ ($t \neq 0$) contains a rational homology ball $B_{n,1}$ and $W^+_0$ is diffeomorphic to $W^+_t$, there is an embedded rational homology ball $B_{n,1}$ in $V_{-4}$.

Let $n \ge 3$ be an even integer. Then $(C \subset \mathcal{W}) \to (Q \in \mathcal{Z})$ is a divisorial extremal neighborhood of type mk1A; cf.~HTU~\cite{Hacking-Tevelev-Urzua-2013}, Urz\'ua~\cite{Urzua-2013}. That is, the map $W_t \to Z_t$ ($t \neq 0$) is a blow-down. On the other hand, $Q$ is a cyclic quotient surface singularity of type $\frac{1}{4}(1,1)$ and $\mathcal{Z} \to \Delta$ is the smoothing of $Q$. So the general fiber $Z_t$ is a rational homology ball $B_{2,1}$, while $W_t$ contains a rational homology ball $B_{n,1}$. Therefore there is an embedding $B_{n,1} \hookrightarrow B_{2,1} \sharp \overline{\mathbb{CP}}^2$.
\end{proof}

\section{A rational blow-up surgery}
\label{section:rational-blowup}

Since a rational blow-down surgery was very successful for constructing many interesting examples of $4$-manifolds, it would be also an intriguing problem to consider the converse surgery, that is, replacing a rational homology ball $B_{p,q}$ by the regular neighborhood $C_{p,q}$, which is called a \emph{rational blow-up surgery}.

Let $X$ be a regular neighborhood of the $\delta$-half linear corresponding to $p^2/(pq-1)$ and let $B_{p,q}$ be a rational homology ball embedded in $X$ which is constructed in the proof of Theorem~\ref{theorem:Bpq} above. Let $Z=(X-B_{p,q}) \cup_{L(p^2,pq-1)} C_{p,q}$ be a smooth $4$-manifold obtained by rationally blowing up of $X$, that is, by replacing the rational homology ball $B_{p,q}$ with a configuration $C_{p,q}$. Suppose that $\widetilde{X}$ is a regular neighborhood of the linear chain in \eqref{equation:linear-chain-p^2/(pq-1)} which is obtained by blowing up appropriately $X$ as in Corollary~\ref{corollary:delta-half->chain} above. Then we have

\begin{proposition}
$Z$ is diffeomorphic to $\widetilde{X}$.
\end{proposition}

\begin{proof}
Let $\widehat{X}$ be a singular complex surface obtained by contracting the linear chain \eqref{equation:Cpq} in $\widetilde{X}$ to the singular point $x \in \widehat{X}$. Let $\pi \colon \widehat{\smash[b]{\mathcal{X}}} \to \Delta$ be the smoothing of $\widehat{X}$ induced by a local smoothing of $x \in \widehat{X}$. Then a general fiber $\widehat{X}_t = \pi^{-1}(t)$ ($t \neq 0$) contains a rational homology ball $B_{p,q}$ and the procedure from $\widetilde{X}$ to $\widehat{X}_t$ is the rational blow-down surgery as we saw in Remark~\ref{remark:Q-bldn=smoothing} above. That is, $\widehat{X}_t$ is obtained from $\widetilde{X}$ by rationally blowing down $C_{p,q}$.

On the other hand, according to the proof of Theorem~\ref{theorem:Bpq} above, if we apply flips appropriately to $\pi \colon \widehat{\smash[b]{\mathcal{X}}} \to \Delta$, we obtain a new deformation $\pi' \colon \widehat{\smash[b]{\mathcal{X}}}' \to \Delta$ such that its central fiber $\pi'^{-1}(0)$ is changed to $X$ but its general fiber is still $\widehat{X}_t$; so $X$ is diffeomorphic to $\widehat{X}_t$. But, if we rationally blow up $X$, or equivalently, if we rationally blow up $\widehat{X}_t$ , then we get again $\widetilde{X}$ as a resulting smooth $4$-manifold, which can be obtained from $X$ by appropriate blowing-ups. Therefore the assertion follows.
\end{proof}


\end{document}